\documentclass[12pt, a4paper,reqno]{amsart}

\usepackage[euler-digits]{eulervm}
\usepackage{graphicx,enumitem}
\usepackage{amsfonts, amsthm, amssymb, amsmath, stmaryrd}
\usepackage{mathrsfs,array}
\usepackage{eucal,times,color,accents}
\usepackage{url}
\usepackage{float}
\usepackage{pbox}
\usepackage[headings]{fullpage}
\usepackage{enumitem}
\usepackage{ytableau}
\usepackage{color}
\usepackage{mathrsfs}
\usepackage{amssymb}
\usepackage{bm}
\usepackage{amssymb}
\usepackage{ulem}
\usepackage{hyperref, cleveref}
\usepackage[all,cmtip]{xy}
\usepackage{comment}

\newcommand{\ncom}{\newcommand}

\ncom{\dho}{\partial}
\ncom{\rar}{\rightarrow}
\ncom{\imply}{\Rightarrow}
\ncom{\lrar}{\longrightarrow}
\ncom{\into}{\hookrightarrow}
\ncom{\onto}{\twoheadrightarrow}
\ncom{\ov}{\overline}
\ncom{\m}{\mbox}
\ncom{\sta}{\stackrel}
\ncom{\invlim}{\varprojlim}
\ncom{\xhat}{\widehat}
\ncom{\vspc}{\vspace{3mm}}
\ncom{\End}{{\cE}nd}
\ncom{\tensor}{\otimes}
\ncom{\al}{\alpha}
\ncom{\cHom}{{\mathcal Hom}}
\ncom{\A}{{\mathbb A}}
\ncom{\comx}{{\mathbb C}}
\ncom{\E}{{\mathbb E}}
\ncom{\F}{{\mathbb F}}
\ncom{\G}{{\mathbb G}}
\ncom{\K}{{\mathbb K}}
\ncom{\Le}{{\mathbb L}}
\ncom{\N}{{\mathbb N}}
\ncom{\p}{{\mathbb P}}
\ncom{\Q}{{\mathbb Q}}
\ncom{\R}{{\mathbb R}}
\ncom{\Z}{{\mathbb Z}}
\ncom{\f}{\dfrac}
\ncom{\wtil}{\widetilde}
\ncom{\ci}{{\mathpzc i}}
\ncom{\cA}{{\mathcal A}}
\ncom{\cB}{{\mathcal B}}
\ncom{\cC}{{\mathcal C}}
\ncom{\cE}{{\mathcal E}}
\ncom{\cF}{{\mathcal F}}
\ncom{\cG}{{\mathcal G}}
\ncom{\cH}{{\mathcal H}}
\ncom{\cI}{{\mathcal I}}
\ncom{\cJ}{{\mathcal J}}
\ncom{\cK}{{\mathcal K}}
\ncom{\cL}{{\mathcal L}}
\ncom{\cM}{{\mathcal M}}
\ncom{\cN}{{\mathcal N}}
\ncom{\cO}{{\mathcal O}}
\ncom{\cP}{{\mathcal P}}
\ncom{\cQ}{{\mathcal Q}}
\ncom{\cR}{{\mathcal R}}
\ncom{\cS}{{\mathcal S}}
\ncom{\cT}{{\mathcal T}}
\ncom{\cU}{{\mathcal U}}
\ncom{\cV}{{\mathcal V}}
\ncom{\cW}{{\mathcal W}}
\ncom{\cX}{{\mathcal X}}
\ncom{\cY}{{\mathcal Y}}
\ncom{\cZ}{{\mathcal Z}}
\ncom{\cSU}{{\mathcal S \mathcal U}}
\ncom{\eop}{{\hfill $\Box$}}
\ncom{\isom}{\cong}

\theoremstyle{plain}
\newtheorem{theorem}{Theorem}

\newtheorem{thm}[theorem]{Theorem}

\newtheorem{corollary}[theorem]{Corollary}
\newtheorem{proposition}[theorem]{Proposition}
\theoremstyle{definition}

\theoremstyle{remark}
\newtheorem{remark}{Remark}
\newtheorem{definition}{Definition}
\newtheorem{eg}{Example}

\long\def\comment#1{}

\begin{document}
	\title{On a relative dependency formula}
	
	\author{Shashi Ranjan Sinha}
	\address{Department of Mathematics, Indian Institute of Technology -- Hyderabad, 502285, India.}
	\email{ma20resch11005@iith.ac.in}
	
	\begin{abstract} 
		Celikbas, Liang and Sadeghi established a one-sided inequality for the relative version of Jorgensen's dependency formula and questioned whether it would be an  equality. In this paper, we show that the inequality can be indeed strict, and prove a relative dependency formula. Along the way, we obtain some bounds on $s(M,N)$, a notion related to the vanishing of relative homology of finitely generated modules $M$ and $N$ over a local ring $R$, under specific assumptions.
	\end{abstract}
	
	\date{\today}
	\maketitle
	\section{Introduction}
	Throughout the paper, $R$ stands for a commutative Noetherian local ring $(R,m,k)$ and all $R\mbox{-}$modules are assumed to be nonzero and finitely generated. For any $R\mbox{-}$modules $M$ and $N$, define $$q(M,N) = sup\{i \geq 0 \mid Tor_i ^{R}  (M,N) \neq 0\}, $$ $$ \text{and} \hspace{4mm} t(M,N) = sup\{depth R_p - depth M_p - depth N_p \mid p  \in Supp(M\otimes_{R} N)\}.$$ If, in addition, $M$ admits a proper resolution, then define $$s(M,N) = sup\{i \geq 0 \mid \mathcal{G}Tor_i ^{R}  (M,N) \neq 0\},$$ where $\mathcal{G}Tor_i ^{R}  (M,N)$ denotes the $i\textsuperscript{th}$ relative homology of $M$ and $N$ (see Section 2 for the definitions).
	
	We say that $M$ and $N$ satisfy the \textit{dependency formula} if $q(M,N) = t(M,N)$. This formula yields the classical Auslander$\mbox{-}$Buchsbaum formula when $N$ is taken to be $k$. Jorgensen \cite{J} proved the dependency formula for the modules $M$ and $N$ provided $q(M,N) < \infty$ and one of the two modules has finite complete intersection dimension. Later on, Celikbas, Liang, and Sadeghi discussed a $\mathcal{G} \mbox{-}$relative version of the dependency inequality, and they proved the subsequent result.
	
	\begin{proposition}\cite[Corollary 3.18]{CLS}\label{s<_t}
		Let $M$ and $N$ be $R\mbox{-}$modules such that $M$ has finite Gorenstein dimension. Then $s(M,N) \leq t(M,N)$.
	\end{proposition}
	
	Additionally, they asked in \cite[Question 3.19]{CLS} whether the claim in the aforementioned proposition would always have equality. The main aim of this note is to provide a complete answer to their question. We show that,
	
	\begin{thm}\label{t <_s+1}
		Let $M$ and $N$ be $R\mbox{-}$modules such that $M$ has finite Gorenstein dimension. If $t(M,N) \neq 1$ then $s(M,N) = t(M,N)$, and if $t(M,N) = 1$, then $s(M,N)\in \{0,1\}$.
	\end{thm}
	
	We further discuss Example \ref{main_eg} which, as shown in Remark \ref{Remark_s<=t-1}, suggests that the strict inequality in Theorem \ref{t <_s+1} may indeed hold.
	
	The notion of the \textit{relative homology} was first defined and studied in \cite{AM} by Avramov and Martsinkovsky. In that paper, they establish several results centered around the relative homology; one of such results is as below. 
	\begin{thm}\cite[Theorem 4.2(a)]{AM} \label{Char_Gdim}
		Let $M$ be an $R\mbox{-}$module which admits a proper resolution. Then for each $r \geq 0$, the following conditions are equivalent.
		\begin{enumerate}
			\item $G\mbox{-}dim_{R} (M) \leq r$;
			\item $\mathcal{G}Tor_i ^{R}  (M,N) = 0$ \:for all $i \geq r+1 $ and for every $R\mbox{-}$module $N$;
			\item $\mathcal{G}Tor_{r+1} ^{R}  (M,N) = 0$ \:for every $R\mbox{-}$module $N$.
		\end{enumerate}
	\end{thm} 
	
	Theorem \ref{Char_Gdim} indicates that $s(M,\mbox{-})$ is bounded above by the Gorenstein dimension of $M$, assuming that the latter is finite. We prove the following theorem, which reveals that this upper bound on $s(M,\mbox{-})$ can be improved under various conditions.
	
	\begin{thm}\label{Improved_bound}
		Let $M$ and $N$ be $R\mbox{-}$modules such that $M$ has finite Gorenstein dimension. Then the following statements hold.
		\begin{enumerate}[label=(\Alph*)]
			\item \label{Suff_rel_vanish} Consider the following conditions:
			\begin{enumerate}[label=(\roman*)]
				\item $N$ satisfies $\widetilde{S_k}$ \:for some $k \leq G\mbox{-} dim_R (M)=r$.
				\item $H\mbox{-} dim(N) < \infty$ and $M$ satisfies $\widetilde{S_k}$ \:for some $k \leq H\mbox{-} dim(N)=r$, where $H\mbox{-} dim(N)$ denotes a homological dimension of $N$ (see Section 2 for details).
			\end{enumerate} If atleast one of the above two conditions hold, then $s(M,N) \leq r-k$.
			\item \label{Suuf_rel_hom} Assume the following conditions hold: 
			\begin{enumerate}[label=(\roman*)]
				\item $G\mbox{-} dim_{R} (M) \leq depthN$, and
				\item $depth (\mathcal{G}Tor_i ^{R} (M,N)) \in \{0, \infty\}$ \:for all $i =1,..., G\mbox{-} dim_{R} (M)$.
				
			\end{enumerate}
			If \:$\widehat{Tor} _i ^{R} (M,N) = 0$ \:for all $i \leq 0$, then $s(M,N) = 0$.
		\end{enumerate}
		
	\end{thm}
	
	In our next result, we aim to generalize Theorem \ref{Char_Gdim} in a way that characterizes projective dimension in terms of absolute homology.
	
	\begin{proposition}\label{Char_Gdim_poistive}
		Let $R$ be a local ring and $M$ be an $R\mbox{-}$ module with finite Gorenstein dimension. Then, for each positive integer $r$, the following statements are equivalent.
		\begin{enumerate}
			\item $G\mbox{-}dim_{R} (M) \leq r$;
			\item $\mathcal{G}Tor_i ^{R}  (M,N) = 0$  for all $i \geq r+1 $, i.e., $s(M,N) \leq r$ for every $R\mbox{-}$module $N$;
			\item $\mathcal{G}Tor_{r+1} ^{R}  (M,N) = 0$ for every $R\mbox{-}$module $N$;
			\item $\mathcal{G}Tor_{r+1} ^{R}  (M,k) = 0$.
		\end{enumerate}
	\end{proposition}

	\subsection*{Acknowledgements} The author thanks Amit Tripathi for helpful comments and suggestions. The author was partially supported by a CSIR senior research fellowship through the grant no. 09/1001(0097)/2021-EMR-I. 
	
	\section{Preliminaries}
	A morphism $\phi : \cA \rar \cB$ of complexes of $R\mbox{-}$modules is called \textit{quasi-isomorphism} if the induced morphism $H_i (\phi) : H_i (\cA) \rar H_i (\cB)$ of homology is bijective for all integers $i$. A resolution of an $R\mbox{-}$module $M$ is a quasi-isomorphism $\beta : \cA \rar M$ where $\cA$ is a complex with $\cA_i =0$ \:for all integer $i <0$ and $M$ is identified as a complex with $M$ in degree zero and $0$ elsewhere. The associated exact sequence to the resolution $\beta$ is $$\beta^+ : \hspace{5mm} \cdots \rar \cA_i \rar \cdots \rar \cA_1 \rar \cA_0 \rar M \rar 0.$$
	
	An $R\mbox{-}$module $M$ can have several different homological dimensions, such as its projective dimension $pd_R(M)$, Gorenstein dimension \cite{AB}, complete intersection dimension \cite{AGP}, lower complete intersection dimension \cite{Gerko}, upper Gorenstein dimension \cite{Veliche} or Cohen-Macaulay dimension  \cite{Gerko}. We collectively denote all these homological dimensions as $H\mbox{-}dim_R(M)$. As recorded in \cite[Theorems 8.6, 8.7]{Av}, these homological dimensions share some common properties. \\
	
	We take a brief look at the notion of Gorenstein dimension \cite{AB}. An $R\mbox{-}$module $G$ is said to be totally reflexive (or of Gorenstein dimension zero) whenever 
	\begin{enumerate}
		\item the natural evaluation map $G \rar G^{**}$ is an isomorphism; and
		\item $Ext^{i} _{R} (G,R) = Ext^{i} _{R} (G^* ,R) =0$ for all positive integers $i$,
	\end{enumerate} 
	
	where $(-)^*$ stands for $\text{Hom}_R
	(-,R)$. \\
	The Gorenstein dimension of an $R\mbox{-}$module $M$, denoted $G\mbox{-}dim_{R}(M)$, is the minimum length of a resolution of $M$ by totally reflexive $R\mbox{-}$modules. We outline some basic properties about Gorenstein dimension.
	
	\begin{thm}\cite{AB}\label{G-dim}
		Let $M$ be an $R\mbox{-}$module. Then the following statements hold.
		\begin{enumerate}[label=(\alph*)]
			\item $G\mbox{-}dim_{R}(M) \leq pd(M)$, and the equality holds if \:$pd(M) < \infty$.
			\item \label{M^*_TR} $G\mbox{-}dim_{R}(M) =0 \implies G\mbox{-}dim_{R}(M^*) =0$.
			\item \label{AB formula} If $G\mbox{-}dim_{R}(M) < \infty$, then $G\mbox{-}dim_{R}(M) + depth(M) = depth(R)$. This equality is known as Auslander-Bridger formula.
			\item \label{Gdim M_p} $G\mbox{-}dim_{R_p} (M_p) \leq G\mbox{-}dim_{R} (M)$ for all $p \in Supp(M)$. In particular, if $M$ is totally reflexive, then $M_p$ is also totally reflexive for all $ p \in \text{Supp}(M)$.
		\end{enumerate}
		
	\end{thm}
	
	A $\mathcal{G} \mbox{-}$approximation of $R\mbox{-}$module $M$ is an exact sequence $0 \rar X \rar G \rar M \rar 0$, where $pd(X) < \infty$ and $G$ is a totally reflexive $R\mbox{-}$ module.\label{G_approx} It follows from \cite[Theorem 3.1]{AM} that any $R\mbox{-}$module $M$ with finite Gorenstein dimension admits a $\mathcal{G} \mbox{-}$approximation as above with $pd(X)= G\mbox{-}dim_{R}(M) -1$.\\
	
	\subsection{Tate and Relative homology}
	The definitions and results that follow are primarily taken from \cite{AM}.	A complete resolution of an $R$-module $M$ is a diagram	$$\mathcal{C} \xrightarrow{\varphi} \mathcal{F} \rar M,$$
	where $\mathcal{C}$ is a totally acyclic complex of free $R$-modules (i.e., an acyclic complex whose dual complex is acyclic as well), $\mathcal{F}$ is a free resolution of $M$, and $\varphi$ is a morphism such that $\varphi_{i}$ is an isomorphism for $i \gg 0$. By \cite[Lemma 3.1]{AM}, an $R$-module $M$ admits a complete resolution if and only if $\text{G-dim}(M) < \infty$. 
	
	Let $M$ be an $R\mbox{-}$module with a complete resolution $\mathcal{C} \rar \mathcal{F} \rar M$. Then, for any $R$-module $N$ and for each integer $i$, the \textit{Tate} homology of $M$ and $N$ is defined as $$\widehat{Tor}_{i} ^{R} (M,N) = H_{i} (\mathcal{C} \otimes_{R} N).$$

	A sequence $\vartheta$ of $R$-modules is referred to as \textit{proper exact} if it satisfies the condition that for every totally reflexive $R$-module $G$, the induced sequence $\text{Hom}_R (G, \vartheta)$ is exact. A proper resolution of an $R\mbox{-}$module $M$ is a resolution $\mathcal{T} \rar M$ by totally reflexive $R\mbox{-}$modules whose associated exact sequence $\mathcal{T}^+$ is proper. Every module with finite Gorenstein dimension admits a proper resolution (see \cite[Lemma 4.1]{AM}).
	
	Let $M$ be an $R\mbox{-}$module which admits a proper resolution $\mathcal{T} \rar M$. Then we define for each $R\mbox{-}$module $N$ and for each integer $i$ the \textit{relative} homology of $M$ and $N$ as $$\mathcal G Tor_{i} ^{R} (M,N) = H_{i} (\mathcal{T} \otimes_{R} N).$$
	
	Below is a list of properties from \cite{AM} that will be referenced throughout the paper.

	\begin{thm}\label{Res_2}

		Let $M$, $N$ and $\{M_i \mid i \in \{1,2,3\}\}$ be $R\mbox{-}$modules. Then the following statements hold true.

		\begin{enumerate}[label=(\alph*)]
			\item Each exact sequence \:$0 \rar M_1 \rar M_2 \rar M_3 \rar 0$ \:with $pd(M_1) < \infty$ is proper exact. In particular, every $\mathcal{G} \mbox{-}$approximation of an $R\mbox{-}$module is proper exact. \label{proper_ses}
			\item \label{Pd+hom} If \:min\hspace{1mm}$\{\text{pd}(M),\text{pd}(N)\} < \infty$, then $\widehat{Tor}_{i} ^{R} (M,N) =0$ \:and \:$\mathcal{G}Tor_i ^{R}  (M,N)  \cong Tor_{i} ^{R} (M,N)$ for all integers $i$.			
			\item If $M$ admits a proper resolution, then
			\begin{equation*}
				\mathcal{G}Tor_i ^{R}  (M,N) =
				\left\{
				\begin{array}{ll}
					M \otimes_{R} N, & \text{if } \:i = 0\\
					0, & \text{if } \:i < 0.
				\end{array}
				\right.
			\end{equation*}
			\item \label{seq_rel} If \:$0 \rar M_1 \rar M_2 \rar M_3 \rar 0$ \:is a proper exact sequence with each module having finite Gorenstein dimension, then there is an exact sequence: 
			\begin{align*}
				\hspace{1cm}\cdots \rar & \hspace{1mm}\mathcal{G}Tor_{i+1} ^{R}(M_3,N) \xrightarrow{\partial_{i+1} ^{\cG}}  \mathcal{G}Tor_i ^{R}(M_1,N) \rar  \mathcal{G}Tor_i ^{R}(M_2,N) \rar \mathcal{G}Tor_i ^{R}(M_3,N) \\ & \rar  \cdots \rar \mathcal{G}Tor_1 ^{R}(M_3,N) \xrightarrow{\partial_{1} ^{\cG}} M_1\otimes_{R} N \rar M_2\otimes N  \rar M_3\otimes N \rar 0.
			\end{align*} 
			\item \label{seq_ART} If \:$G\mbox{-}dim_{R} (M) < \infty$, then there exists an exact sequence:
			\begin{align*}
				\hspace{1cm} \cdots \rar & \widehat{Tor}_i ^{R}(M,N) \rar  Tor_i ^{R}(M,N) \xrightarrow{{\epsilon_{i}^{\cG}(M,N)}}   \mathcal{G}Tor_i ^{R}(M,N) \rar 
				\cdots \rar \mathcal{G}Tor_{2} ^{R}(M,N) \\ & \rar \widehat{Tor}_1 ^{R}(M,N) \rar  Tor_1 ^{R}(M,N) \xrightarrow{{\epsilon_{1}^{\cG}(M,N)}}   \mathcal{G}Tor_1 ^{R}(M,N) \rar 0.
			\end{align*} 
		\end{enumerate}
	\end{thm}
	
	We recall the result of Jorgensen on the dependency formula.
	\begin{theorem}\cite[Theorem 2.2]{J}\label{Jorg_dependency}
		Let $M$ and $N$ be $R\mbox{-}$modules such that $M$ has finite complete intersection dimension. If $q(M,N) < \infty$, then $M$ and $N$ satisfy the dependency formula.
	\end{theorem}
	
	The following definition, which is notably weaker than the well-known Serre condition $S_k$ \cite[Page 3]{EG}, will be used later.
	\begin{definition}
		An $R\mbox{-}$module $M$ is said to satisfies $\widetilde{S_k}$ condition for some $k \geq 0$ if $$depthM_p \geq min\{k,depthR_p\}, \:\text{for all} \: p \in Supp(M).$$
	\end{definition}
	We conclude this section by recording a result due to Celikbas, Liang and Sadeghi, which establishes the \textit{depth formula} as follows.
	
	\begin{proposition}\cite[Corollary 3.17]{CLS}\label{depth_formula}
		Let $M$ and $N$ be $R\mbox{-}$modules such that $G\mbox{-} dim_R (M) < \infty$ and $\widehat{Tor}_i ^{R}(M,N) = 0 \:\: \text{for all} \:\: i \leq 0$. Set $s=s(M,N)$. If $s=0$ or $\text{depth}(\cG{Tor}_{s}  ^{R}(M,N)) \leq 1$, then the following equality holds
		$$\text{depth} M + \text{depth} N = \text{depth}R + \text{depth}(\cG{Tor}_{s}  ^{R}(M,N)) - s.$$
	\end{proposition}
	
	\section{Proof of the main results}
	In this section, we discuss the proof of our main results. We also present Example \ref{main_eg} at the end, which was stated in the introduction.
	
	\begin{proof}[Proof of Theorem \ref{t <_s+1}]
		In view of Proposition \ref{s<_t}, it is enough to show that $t(M,N) \leq 1$ if $s(M,N)=0$ and $t(M,N) \leq s(M,N)$ if $s(M,N) \geq 1$ . Assuming $M$ is totally reflexive, the assertion follows from the Auslander-Bridger formula \ref*{G-dim}\ref{AB formula} and Theorem \ref*{G-dim}\ref{Gdim M_p}. So we may assume that  $G\mbox{-}dim_{R} (M) \geq 1$. Consider a $\mathcal{G} \mbox{-}$approximation of $M$
		
		\begin{equation}\label{E_MG}
			0 \rar X \rar G \rar M \rar 0
		\end{equation}
		where $pd(X) < \infty$ and $G$ is a totally reflexive $R\mbox{-}$module \ref{G_approx}.\\
		
		Note that the sequence (\ref{E_MG}) is proper exact. Also $\mathcal{G}Tor_i ^{R}  (G,N) = 0$ for all $i \geq 1$ by Theorem \ref{Char_Gdim}. Therefore on applying Theorem \ref*{Res_2}\ref{seq_rel} to the sequence \ref{E_MG}, we get \begin{equation*}
			s(X,N) =
			\left\{
			\begin{array}{ll}
				s(M,N)-1, & \text{if } s(M,N) \geq 1\\
				0, & \text{if } s(M,N)=0.
			\end{array}
			\right.
		\end{equation*}
		Since $pd(X) < \infty$, it follows from Theorem \ref*{Res_2}\ref{Pd+hom} that $Tor_i ^{R}  (X,N) \cong \mathcal{G}Tor_i ^{R}  (X,N)$\:for all $i \geq 0$, whence \begin{equation*}
			q(X,N) =
			\left\{
			\begin{array}{ll}
				s(M,N)-1, & \text{if } s(M,N) \geq 1\\
				0, & \text{if } s(M,N)=0.
			\end{array}
			\right.
		\end{equation*} \\ 
		Now for any $p \in Supp(M\otimes_{R} N)$, there are two possible cases\\
		\begin{enumerate}
			\item $p \notin Supp(X)$. Then $M_p \cong G_p$, is a totally reflexive module by Theorem \ref*{G-dim}\ref{Gdim M_p}. The Auslander-Bridger formula \ref*{G-dim}\ref{AB formula} thus implies $depth R_p - depth M_p - depth N_p \leq 0$. 
			\item $p \in Supp(X)$. Using Theorem \ref{Jorg_dependency} for the pair $(X,N)$, we get
			\begin{equation*}
				depth R_p - depth X_p - depth N_p \leq
				\left\{
				\begin{array}{ll}
					s(M,N)-1, & \text{if } s(M,N) \geq 1\\
					0, & \text{if } s(M,N)=0,
				\end{array}
				\right.
			\end{equation*} \\
			which on equating further gives
			
			\begin{equation}\label{X_P -1}
				depth X_p - 1 \geq
				\left\{
				\begin{array}{ll}
					depth R_p - depth N_p - s(M,N), & \text{if } s(M,N) \geq 1\\
					depth R_p - depth N_p - 1, & \text{if } s(M,N)=0.
				\end{array}
				\right.
			\end{equation}\\
			Since $G_p$ is total reflexive, $depth G_p = depth R_p \geq depth R_p - depth N_p - s(M,N)$. Localizing the sequence (\ref{E_MG}) at $p$ and then on applying the depth lemma to the induced sequence, we obtain that $depth M_p \geq depth X_p - 1$. Thus it follows from the inequality \ref{X_P -1} that \begin{equation*}
				depth R_p - depth M_p - depth N_p \leq
				\left\{
				\begin{array}{ll}
					s(M,N), & \text{if } s(M,N) \geq 1\\
					1, & \text{if } s(M,N)=0.
				\end{array}
				\right.
			\end{equation*}
		\end{enumerate} 
		
		This proves the desired claim.
	\end{proof}
	
	\begin{remark}\label{Remark_s<=t-1}
		For any totally reflexive $R\mbox{-}$module $M$, it follows from Theorem \ref*{G-dim}\ref{Gdim M_p} that $M_p$ is totally reflexive as well for all $p \in \text{Supp(M)}$ and hence $s(M,k)=t(M,k)=0$. We shall also see in the example \ref{main_eg} that the residue field $k$ of a non-regular Gorenstein local ring satisfies  $s(k,k)=t(k,k)-1=0$. This shows that the strict inequlity can occur in the above theorem.
	\end{remark}
	
	\begin{proof}[Proof of \ref*{Improved_bound}\ref{Suff_rel_vanish}]
		Notably the statement and the idea of its proof is inspired from the result \cite[A.5]{CDK} of Celikbas, Dey and Kobayashi. Let $p \in Supp(M\otimes_{R} N)$. In view of Proposition \ref{s<_t}, it suffices to show that $$depthM_p + depthN_p \geq depthR_p - (r-k).$$ 
		
		We first establish the claim by assuming the condition in $(i)$. Since $N$ satisfies $\widetilde{S_k}$, $depth N_p \geq min\{k, depth R_p\}$. There is nothing to establish if $depthR_p \leq k$. So assume that  $depthR_p \geq k$. By Theorem \ref*{G-dim}\ref{Gdim M_p} $G\mbox{-} dim_{R_p} (M_p) \leq r$, therefore $depth M_p \geq depth R_p - r$ whence $$depthM_p + depthN_p \geq depthM_p+ depthR_p - r \geq depthR_p -r+k. $$ 
		The proof of the claim, based on the condition in (ii), follows a similar approach as discussed above. It employs the (in)equalities $H\mbox{-}dim_{R}(N) + depth(N) = depth(R)$ and $H\mbox{-} dim_{R_p} (N_p) \leq H\mbox{-} dim_{R} (N)$ for all $p \in \text{Supp(N)}$ (see \cite[Theorem 8.7]{Av}). 
	\end{proof}
	
	The subsequent corollary is a special case of the statement \ref*{Improved_bound}\ref{Suff_rel_vanish}.
	
	\begin{corollary}
		Let $M$ and $N$ be $R\mbox{-}$modules with finite Gorenstein dimension $r$. If $N$ satisfies $\widetilde{S_r}$, then $s(M,N) = 0$.
	\end{corollary}
	
	However the converse of the above corollary need not be true. For instance, the residue field $k$ in the example \ref{main_eg} has Gorenstein dimension one and $s(k,k)=0$ but it does not satisfy $\widetilde{S_1}$ condition because it is a torsion module. \\
	
	We now present the proof for statement \ref*{Improved_bound}\ref{Suuf_rel_hom}, which is similar to that of \cite[Lemma 3.6]{CST}.
	
	\begin{proof}[Proof of  \ref*{Improved_bound}\ref{Suuf_rel_hom}]
		
		Let $s(M,N)=r$. We note that $G\mbox{-} dim (M) \geq r$. If $G\mbox{-} dim (M)=0$ or $r=0$, there is nothing to establish. Therefore, it is reasonable to assume $G\mbox{-} dim (M) \geq r \geq 1.$ The assumption $(ii)$ suggests that $depth (\mathcal{G}Tor_r ^{R} (M,N))=0$. Hence the depth formula in Proposition \ref{depth_formula} yields the following:
		
		\begin{align*}
			depthM + depthN & = depth(\mathcal{G}Tor_r ^{R} (M,N)) + depthR -r \\
			& = depthR -r  < depthR.
		\end{align*}
		
		This implies $G\mbox{-} dim (M) > depthN$, which contradicts our assumption $(i)$. Hence we must have $r=0$. 
	\end{proof}
	
	The next example shows that the vanishing of non-positive Tate homologies in statement \ref*{Improved_bound}\ref{Suuf_rel_hom} is not necessary. 
	
	\begin{eg}\cite[Example 3.3]{CLS}
		Let $R = k[[x,y]]/(xy)$, where $k$ is a field, $M=R/(x)$ and $N=R/(y)$. Then $M$ and $N$ are totally reflexive $R\mbox{-}$modules so that $s(M,N)=0$ and $G\mbox{-} dim (M)=0 < 1= depthN$. But one can check that $\widehat{Tor} _i ^{R} (M,N) \cong k$\: for all negative integers $i$.
	\end{eg} 
	
	Before we discuss the proof of Proposition \ref{Char_Gdim_poistive}, it is important to note that should $r=0$, the conclusion of the assertion may not necessarily be true. For instance, consider the residue field $k$ of the local ring $R$ in the example \ref{main_eg}. The first relative homology $\mathcal G Tor_{1} ^{R} (k,k)$ vanishes while $G\mbox{-}dim _{R} (k) =1.$ 
	
	\begin{proof}[Proof of Proposition \ref{Char_Gdim_poistive}] In view of Theorem \ref{Char_Gdim}, it is enough to show the implication $(4) \implies (1)$. Assume $\mathcal G Tor_{r+1} ^{R} (M,k) = 0$ for some $r\geq 1$. Consider a $\cG$-approximation of $M$  \begin{align} \label{seq_g_approximation} 
			0 \rar X \rar G \rar M \rar 0
		\end{align} where $pd(X) < \infty$ and $G$ is a totally reflexive $R\mbox{-}$ module. 
		
		Theorem \ref{Res_2}\ref{proper_ses} suggests that \eqref{seq_g_approximation} is a proper exact sequence. So for every non-negative integer $r,$ the long exact sequences in Theorems \ref*{Res_2}\ref{seq_rel} and \ref*{Res_2}\ref{seq_ART} induce the following maps 
		\begin{equation*}
			\begin{gathered}
				\xymatrix@C-=1.2cm@R-=2.2cm{Tor_{r+1}^R(M,k) \ar[r]^{\partial^R_{r}} \ar[d]^{\epsilon_{r+1}^{\cG}(M,k)}& Tor_{r}^R(X,k)  \ar[d]^{{\epsilon_{r}^{\cG}(X,k)}}\\ \cG Tor_{r+1}^R(M,k) \ar[r]^{\partial_r^{\cG}} & \cG Tor_{r}^R(X,k)  }
			\end{gathered} 
		\end{equation*} Since $pd(X)< \infty$, it follows from Theorem \ref*{Res_2}\ref{Pd+hom} that the right vertical map ${\epsilon_{r}^{\cG}(X,k)}$ is an isomorphism for each $r \geq 1$. As $r\geq 1,$ and $G$ is totally reflexive so by Theorem \ref{Char_Gdim}, $\cG Tor_{r+1}^R(G,k) = \cG Tor_{r}^R(G,k) = 0$. This means that  $\partial_r^{\cG}$ is an isomorphism. Thus, the hypothesis implies that $Tor_r^R(X,k) = 0$ whence $pd(X) \leq r-1.$ This shows that $G\mbox{-}dim_{R} (M) \leq  r.$
	\end{proof}
	
	It has been observed in \cite[Remark 4.1.4]{AM} that a module which admits a proper resolution may not necessarily have finite Gorenstein dimension. Our next observation, however, shows that this could be the case over a certain class of rings, provided that $s(-,-)$ is finite.
	
	\begin{remark}\label{Golod}
		Let $R$ be a Golod ring which is not a hypersurface, and let $M$ be an $R\mbox{-}$module which admits a proper resolution. Assume that $s(M,N) < \infty$ for some $R\mbox{-}$module $N$. By \cite[Example 3.5.2]{AM}, if $\mathcal{F} \rar M$ is a minimal free resolution, then it is a proper resolution as well. Hence by definition $Tor_i ^{R}  (M,N) = \mathcal{G}Tor_i ^{R}  (M,N)$ \:for all $i \geq 0$, that is, $q(M,N) = s(M,N) < \infty$. The assertion thus follows from \cite[Theorem 3.1]{J}. In particular,  $q(M,N)=s(M,N)=t(M,N)$ by \cite[Theorem 2.2]{J}.
	\end{remark}
	
	Now we come to the example which we advertised in the beginning.

	\begin{eg}\label{main_eg}
		Let $(R,m,k)$ be a one dimensional Gorenstein local ring which is not regular. Then \cite[Theorem 3.2.10]{BH} gives us the isomorphism $Ext^{1} _{R} (k,R) \cong k$. Dualizing $0 \rar m \rar R \rar k \rar 0$ gives an exact sequence 
		\begin{equation}\label{Standard_dual}
			0 \rar R \rar m^* \rar k \rar 0.
		\end{equation}
		
		The maximal ideal $m$ is a totally reflexive $R\mbox{-}$module, and so it follows from Theorem \ref*{G-dim}\ref{M^*_TR} $m^*$ is totally reflexive as well . Thus there exists an element $x \in m$ which is regular on $m^*$ and $R$. Tensoring the sequence \ref{Standard_dual} with $R/x$ we get
		
		$$0 \rar Tor_1 ^{R} (k,\overline{R}) \rar \overline{R} \rar \overline{m^*} \rar \overline{k} \rar 0,$$
		where $\overline{(-)}$ denotes $(-) \otimes_{R} R/x$. We note that $\overline{k} =k$ as $x \in ann(k)$. Thus we have the following diagram of homomorphisms 
		
		\begin{equation*}
			\begin{gathered}    	
				\xymatrix{& m^* \ar[r]^{\cong} \ar[d] & m^* \ar[d]^{x} & \\0 \ar[r] & R \ar[r] \ar[d]  &   m^* \ar[r] \ar[d] & k \ar[r] \ar[d] & 0 \\ 0 \ar[r] & K \ar[r] &  \overline{m^*} \ar[r] & k \ar[r] & 0}
			\end{gathered}
		\end{equation*}
		where $K$ is the kernel of the map $\overline{m^*} \rar k$. The diagram shows that $m^*$ can be viewed as an submodule of $R$ and hence as an ideal of $R$. We note that the sequence \ref{Standard_dual} is proper exact by Theorem \ref*{Res_2}\ref{proper_ses}. Therefore on applying Proposition \ref*{Res_2}\ref{seq_rel} on the sequence \ref{Standard_dual} with $k$ and using Theorem \ref{Char_Gdim} yields an exact sequence $$ 0 \rar \mathcal{G}Tor_{1} ^{R} (k,k) \rar k \rar m^* \otimes_{R} k \rar k \rar0.$$
		
		This shows that $\mu(m^*) \leq 2$, where $\mu(M)$ denotes the minimal number of generators of an $R-$module $M$. If $\mu(m^*)=1$, that is $m^*$ is a principal ideal, then $m^*$ is free $R\mbox{-}$module whence $m$ is free, contradicting the non$\mbox{-}$regularity condition on $R$. Hence $\mu(m^*)=2$ which implies $\mathcal{G}Tor_{1} ^{R} (k,k) =0$.
	\end{eg}

\end{document}